%% file: A_generalization_of_the_notion_of_helix_-_TJM_v3.tex
\documentclass{math}
\usepackage{hyperref}
\hypersetup{
colorlinks=true,
urlcolor=blue,
citecolor=blue}
\usepackage[all]{xy,xypic}
\usepackage{amsfonts,amssymb,amsmath,amsgen,amsopn,amsbsy,theorem,graphicx,epsfig}
\usepackage{eufrak,amscd,bezier,latexsym,mathrsfs,enumerate,multirow}
\usepackage[utf8]{inputenc}\usepackage[english]{babel}
\usepackage[dvipsnames]{xcolor}
\usepackage[pagewise]{lineno}

\yil{}
\vol{}
\fpage{}
\lpage{}
\doi{}

\title{A generalization of the notion of helix}

\author[LUCAS and ORTEGA-YAGÜES]{
\textbf{Pascual LUCAS$^{1}$\thanks{Correspondence: plucas@um.es}, José-Antonio ORTEGA-YAGÜES$^{2}$}\\
$^{1}$Departamento de Matemáticas, Universidad de Murcia, Campus de Espinardo, 30100 Murcia, Spain,\\ ORCID iD: https://orcid.org/0000-0002-4354-9736\\
$^{2}$Departamento de Matemáticas, Universidad de Murcia, Campus de Espinardo, 30100 Murcia, Spain,\\ ORCID iD: https://orcid.org/0000-0001-9521-1051\\ [1.8em]

\rec{.201}
\acc{.201}
\finv{..201}
}

\amssayisi{2010 {\itshape AMS Mathematics Subject Classification:} 53A04
\newline
\vspace{-15mm}
\begin{center}
\raisebox{-17ex}[0ex][0ex]{~~ \raisebox{.5ex}[0ex][0ex]{\footnotesize  This work is licensed under a Creative Commons Attribution 4.0 International License.}}
\end{center}}

\input{math.tex}

\newtheorem{theo}{Theorem}[section]
\newtheorem{prop}[theo]{Proposition}
\newtheorem{defi}[theo]{Definition}

\def\svf{$F$-constant}
\def\<{\left<}
\def\>{\right>}
\def\R#1{\ensuremath{\mathbb{R}^{#1}}}
\def\X{\ensuremath{\mathfrak{X}}}

\begin{document}

\maketitle

\begin{abstract}
In this paper we generalize the notion of helix in the three-dimensional Euclidean space, which we define as that curve $\alpha$ for which there is an \svf{} vector field $W$ along $\alpha$ that forms a constant angle with a fixed direction $V$ (called an axis of the helix). We find the natural equation and the geometric integration of helices $\alpha$ where the \svf{} vector field $W$ is orthogonal to its axis.

\keywords{helix; osculating helix; normal helix; rectifying helix; Darboux vector; \svf{} vector field}
\end{abstract}

\section{Introduction}
\label{s:setup}

Let $\alpha:I\to\R3$ be a differentiable unit speed curve and let $\X(\alpha)$ denote the set of differentiable vector fields along the curve. Let $F(s)=\{F_1(s),F_2(s),F_3(s)\}$ be a moving orthonormal frame along $\alpha$, so we can see $F$ as a differentiable map $F:I\to SO(3)$. It is easy to show (see e.g. \cite[pp. 43--45]{Kreyszig}) that there exists a unique vector field $D_F(s)$ along $\alpha$ satisfying the equations
\begin{equation}\label{D-eqs}
F_1'=D_F\times F_1,\quad F_2'=D_F\times F_2,\quad F_3'=D_F\times F_3,
\end{equation}
where $()'$ is the usual derivative in $\R3$ and $\times$ stands for the cross product. The vector field $D_F$ is called the \emph{Darboux vector associated to the frame $F$}.

From a physical point of view, along the curve $\alpha$ we have two coordinate systems associated to two references: (a) One associated to the frame $F$, that may be imagined as being fixed on the curve. This system rotates and is thus accelerating, it is a non-inertial frame. (b) The other is the canonical rectangular coordinates $(x_1,x_2,x_3)$ in $\R3$, associated to the usual canonical basis $\{e_1,e_2,e_3\}$, that is fixed to the space and is an inertial frame. Obviously, given a vector field $W$ along $\alpha$, the variation in $s$ of its coordinate functions in the two coordinate systems is not the same. In fact, they are related by the following equation:
\[
W'=\frac{d_r}{ds}(W)+D_F\times W,
\]
where $\frac{d_r}{ds}(W)$ denotes the rate of change of $W$ as observed in the rotating coordinate system (i.e., in the frame $F$). The above equation is usually called the \emph{Transport Theorem} in analytical dynamics (see \cite[p. 11]{SJ}). If we think of $W$ as a curve in the 3-space, then the above formula tells us that its (absolute) velocity in the inertial frame $\{e_1,e_2,e_3\}$ is equal to the velocity relative to the moving (rotating) frame $\{F_1,F_2,F_3\}$ plus the velocity of the rotating coordinate system itself.

Motivated by the equations (\ref{D-eqs}) and the Transport Theorem, we introduce the following definition.
\begin{defi}\label{F-constant}
A vector field $W$ along $\al$ is said to be constant with respect to the frame $F$ (or \emph{\svf{} vector field}) if $\frac{d_r}{ds}(W)=0$ (or, equivalently, $W'=D_F\times W$).
\end{defi}
The set of \svf{} vector fields along $\alpha$ will be denoted by $\X_F(\alpha)$. In \cite{Santalo67}, this kind of vector fields are said to be invariably attached to every point of the curve. If we think of the curve $\alpha$ as the trajectory of a rigid body, then a \svf{} vector can be imagined as a body-fixed vector, \cite{SJ}. The following result is a straightforward computation.

\begin{prop}
The following properties of \svf{} vector fields hold:\vspace{-\topsep}
\begin{enumerate}\itemsep0pt
\def\labelenumi{$\theenumi)$}
\item If $W$ is a \svf{} vector field, then $W$ has constant length.
\item If $W_1$ and $W_2$ are \svf{} vector fields, then $W_1+W_2$ is a \svf{} vector field.
\item Let $W$ be a nowhere zero \svf{} vector field and $h$ a differentiable function, then $hW$ is a \svf{} vector field if and only if $h$ is a constant function.
\item $W$ is a \svf{} vector field if and only if $W=a_1F_1+a_2F_2+a_3F_3$, for certain constants $a_i\in\R{}$. Hence, $\X_F(\alpha)$ is a 3-dimensional real vector space.
\end{enumerate}
\end{prop}

A special moving frame is the Frenet frame $F=\{T_\alpha,N_\alpha,B_\alpha\}$, satisfying the well-known Frenet-Serret equations
\begin{equation}\label{FS-eq}
T_\alpha'(s)=\kappa_\al(s)\,N_\alpha(s),\quad
N_\alpha'(s)=-\kappa_\al(s)\,T_\alpha(s)+\tau_\al(s)\,B_\alpha(s),\quad
B_\alpha'(s)=-\tau_\al(s)\,N_\alpha(s),
\end{equation}
where $\kappa_\alpha$ and $\tau_\alpha$ stand for the curvature and torsion functions, respectively. The vector fields $T_\alpha,N_\alpha,B_\alpha$ are trivial examples of \svf{} vector fields. Throughout this paper we will assume that our curves are non-planar, i.e., with non-zero torsion. For this frame, the Darboux vector is simply denoted by $D_\al$ and is given by $D_\al=\tau_\alpha T_\alpha+\kappa_\alpha B_\alpha$. This vector can be interpreted as the angular velocity of the Frenet frame as a whole. In fact, the rate of change of the frame $\{T_\alpha,N_\alpha,B_\alpha\}$ with $s$ can be characterized as an instantaneous rotation about the vector $D_\al$, with angular velocity equal to the \emph{total curvature} specified by
\[
\omega=|D_\al|=\sqrt{\kappa_\alpha^2+\tau_\alpha^2}.
\]

Let $W$ be a nonzero differentiable vector field along the curve $\alpha$. $W$ is said to be a normal vector field if $W(s)$ belongs to the normal plane for every $s$; similarly, we have rectifying or osculating vector fields, depending on whether $W(s)$ belongs to the rectifying or osculating plane, respectively.

\def\scur{}
The term helix (or general helix, or cylindrical helix, or curve of constant slope) has traditionally been used to define curves whose tangent vector field $T_\al$ forms a constant angle with a fixed direction in the 3-space, \cite[p. 33]{Struik61}. The concept of helix has been extended by considering vector fields other than the tangent vector field $T_\al$, such as the principal normal vector field $N_\al$ (thus giving rise to slant helices, \cite{IT04,LO16a}). In this paper, we propose a new extension of the concept by considering \svf{} vector fields.
\begin{defi}\label{helix}
A curve $\alpha$ is said to be a \emph{helix} if there exists a \svf{} vector field $W$ along $\alpha$ that forms a constant angle with a fixed direction $V$, called an \emph{axis} of the helix.
\end{defi}
Without loss of generality, we can assume that $W$ is a unit vector field. In the particular case that $W$ is a normal (osculating or rectifying, resp.) vector field then $\alpha$ is called a \scur{} normal (osculating or rectifying, resp.) helix. Note that we recover the notion of cylindrical helix or slant helix when the \svf{} vector field $W$ is given by $T_\al$ or $N_\al$, respectively. Note also that the term osculating helix has been previously used to refer to the circular helix passing through a point of a curve, having the same tangent, curvature vector and torsion, \cite[p. 42]{Struik61}.

In this paper, we solve the following problem:
\begin{quote}
\itshape
How are characterized the helices $\alpha$ when the \svf{} vector field $W$ is orthogonal to its axis $V$?
\end{quote}
Note that when the vector field $W$ is $T_\alpha$, $N_\alpha$ or $B_\alpha$, then the curve $\alpha$ is nothing but a plane curve or a cylindrical helix. Therefore, in the following sections we will address the question when the vector field $W$ is expressed as a linear combination of at least two vector fields of the Frenet frame.

\section{Normal helices}

\subsection{Natural equation of normal helices}
\label{s:nh1}

Let $\alpha$ be a nonplanar curve in $\R3$ with Frenet apparatus $\{\kappa_\alpha,\,\tau_\alpha;T_\alpha,N_\alpha,B_\alpha\}$, and assume that $\alpha$ is a normal helix with axis $V$, $V$ being a constant vector. Suppose that there is a nonzero constant angle $\theta\in(-\pi/2,\pi/2)$ such that $W=\cos\theta\, N_\alpha+\sin\theta\,B_\alpha$ is orthogonal to $V$. Hence we can write
\begin{equation}\label{V}
V=\lambda\,T_\alpha+\mu(\sin\theta\, N_\alpha-\cos\theta\,B_\alpha),
\end{equation}
for certain differentiable functions $\lambda$ and $\mu$. By taking derivative in (\ref{V}) we get
\begin{align}
\lambda'-\mu\sin\theta\,\kappa_\alpha &=0, \label{eqV'1}\\
\sin\theta\,\mu'+\lambda\,\kappa_\alpha+\mu\cos\theta\,\tau_\alpha &=0, \label{eqV'2}\\
-\cos\theta\,\mu'+\mu\sin\theta\,\tau_\alpha &=0.\label{eqV'3}
\end{align}
From equation (\ref{eqV'3}) we have
\begin{equation}\label{mu}
\mu=e^{\tan\theta\,\int\tau_\alpha},
\end{equation}
that jointly with (\ref{eqV'2}) leads to
\begin{equation}\label{a}
\lambda=-\sec\theta\,\rho\, e^{\tan\theta\,\int\tau_\alpha},
\end{equation}
where $\rho=\tau_\alpha/\kappa_\alpha$ is called the Lancret curvature. Finally, putting equations (\ref{mu}) and (\ref{a}) in (\ref{eqV'1}) yields
\begin{equation}\label{8'}
-\sec\theta(\rho'+\tan\theta\,\tau_\alpha\,\rho)=\sin\theta\,\kappa_\alpha,
\end{equation}
and then
\begin{equation}\label{8''}
-\rho'=\sin\theta\cos\theta\,\kappa_\alpha+\tan\theta\,\tau_\alpha\,\rho.
\end{equation}
Therefore
\begin{equation}\label{ecnaturalnormal}
  \frac{\kappa_\alpha}{\cos^2\theta\,\kappa_\alpha^2+\tau^2_\alpha}\left(\frac{\tau_\alpha}{\kappa_\alpha}\right)'=-\tan\theta.
\end{equation}
Conversely, let $\alpha$ be a curve in $\R{3}$ satisfying equation (\ref{ecnaturalnormal}), for a nonzero constant $\theta\in(-\pi/2,\pi/2)$. Let $V$ be the vector field given in (\ref{V}), where $\mu$ and $\lambda$ are defined by equations (\ref{mu}) and (\ref{a}), respectively. Then equations (\ref{eqV'2}) and (\ref{eqV'3}) are satisfied. On the other hand, from equation (\ref{ecnaturalnormal}) we easily get (\ref{8'}), and then we deduce (\ref{eqV'1}) since
\[
\mu\sin\theta\;\kappa_\al=-\sec\theta\big(\rho'\mu+\tan\theta\,\tau_\alpha\,\rho\,\mu\big)=\lambda'.
\]
Hence, there is a constant $\theta$ such that the vector field $\cos\theta\,N_\alpha+\sin\theta\,B_\alpha$ is orthogonal to a constant direction $V$, so $\alpha$ is a normal helix.

Summing up, we have shown the following result (we also include the case $\theta=0$).
\begin{theo}
Let $\alpha$ be a nonplanar arclength parametrized curve in $\R3$, with curvature $\kappa_\alpha>0$ and torsion $\tau_\alpha$. Then $\alpha$ is a normal helix (with $W$ orthogonal to $V$) if and only if the following equation holds
\[
\frac{\kappa_\alpha}{\cos^2\theta\,\kappa_\alpha^2+\tau^2_\alpha}\left(\frac{\tau_\alpha}{\kappa_\alpha}\right)'=-\tan\theta,
\]
for a certain constant $\theta\in(-\pi/2,\pi/2)$.
\end{theo}

\subsection{Geometric interpretation of normal helices}
\label{s:nh2}

Let $\alpha$ be a normal helix with axis $V$, and assume that $V$ is orthogonal to the \svf{} vector field $W=\cos\theta\,N_\alpha+\sin\theta\,B_\alpha$, $\theta$ being a constant. Let us consider $C_{\alpha,V}$ the cylinder parametrized by $X(t,z)=\alpha(t)+z\,V$, then we have
$N=X_t\times X_z=T_\alpha\times V=\cos\theta\,N_\alpha+\sin\theta\,B_\alpha$, up to a sign.
This shows that the principal normal vector field $N_\alpha$ of the curve $\alpha$ makes a constant angle $\theta$ with the unit vector field $N$ normal to the cylinder $C_{\alpha,V}$. It is not difficult to see that this condition characterizes the normal helices.

Let $M=C_{\beta,V}$ be a general cylinder parametrized by $X(t,z)=\beta(t)+zV$, where $\beta$ is a unit planar curve and $V$ is a unit vector orthogonal to that plane. If $\{T_\beta,N_\beta\}$ is the Frenet frame of $\beta$, assume that the unit normal vector to the cylinder is given by $N(t,z)=T_\beta(t)\times V=N_\beta(t)$. Let us assume that $\alpha(s)=X(t(s),z(s))$, $s\in I$, is an arclength parametrized curve in $M$ such that $N_\al$ makes a constant angle $\theta$ with $N$. A straightforward computation yields
\begin{align}\label{TNB-cil}
 T_\alpha(s)&=\cos\phi(s)\,T_\beta(t(s))+\sin\phi(s)\,V, &&\text{(a)}\nonumber\\
 N_\alpha(s)&=\sin\theta\,\big(-\sin\phi(s)\,T_\beta(t(s))+\cos\phi(s)\,V\big)+\cos\theta\, N,&&\text{(b)}\\
 B_\alpha(s)&=-\cos\theta\,\big(-\sin\phi(s)\,T_\beta(t(s))+\cos\phi(s)\,V\big)+\sin\theta\, N, &&\text{(c)}\nonumber
\end{align}
where $\phi\in\mathcal{C}^\infty(I)$ is a differentiable function with $t'(s)=\cos\phi(s)$ and $z'(s)=\sin\phi(s)$. It is easy to see that
\begin{equation}\label{TNB-cil2}
V=\sin\phi(s)\,T_\al(s)+\cos\phi(s)\big(\sin\theta\,N_\al(s)-\cos\theta\,B_\al(s)\big),
\end{equation}
and then we can define the $F$-constant vector field $W=\cos\theta\,N_\alpha+\sin\theta\,B_\alpha$ satisfying $\<W,V\>=0$, showing that $\al$ is a normal helix. Therefore, we have proved the following result.

\begin{theo}\label{th3}
A curve $\alpha$ in $\R3$ is a normal helix with axis $V$ if and only if $\alpha$ lies on a cylinder $C$ and its principal normal vector field makes a constant angle with the normal vector field to the cylinder.
\end{theo}

Note that when $\theta=0$ then $\alpha$ is a geodesic of the cylinder, and so it is a cylindrical helix. Hence, Theorem \ref{th3} is an extension of the well known theorem of Lancret that characterizes the cylindrical helices as the geodesics of the cylinders.

We finish this section with the following result about curves $\al(s)$ in a cylinder $C_{\beta,V}$. By taking derivative in equations (a) and (c) of (\ref{TNB-cil}) we get
\begin{align*}
\kappa_\alpha(s) N_\alpha(s) &=\phi'(s)(-\sin\phi(s)\,T_\beta(t(s))+\cos\phi(s)\,V)+\cos^2\phi(s)\kappa_\beta(t(s))\,N_\beta(t(s)),\\
-\tau_\alpha(s) N_\alpha(s)&= \phi'(s)\cos\theta(\cos\phi(s)\,T_\beta(t(s))+\sin\phi(s)\,V)+\\
&\kern1em+\cos\phi\kappa_\beta(-\sin\theta\, T_\beta(t(s))+\cos\theta\sin\phi\, N_\beta(t(s))).
\end{align*}
These two equations lead to the following result.
\begin{prop}\label{hn-cil}
Let $\alpha(s)=X(t(s),z(s))$ be an arclength parametrized curve in a cylinder $C_{\beta,V}$. The principal normal vector field $N_\alpha$ makes a constant angle $\theta$ with the normal to the cylinder if and only if there is a differentiable function $\phi$ such that the following equations hold
\begin{align}
  t'(s)&=\cos\phi(s),\label{ecuac cilind1} \\
  z'(s)&=\sin\phi(s), \label{ecuac cilind2}\\
  \phi'(s)&=\tan\theta\,\cos^2(\phi(s))\,\kappa_\beta(t(s)).\label{ecuac cilind3}
\end{align}
Moreover, the curvature and torsion of $\alpha$ are given by
\begin{equation}\label{curvaturas cilindro}
  \kappa_\alpha(s)=\frac{\cos^2\phi(s)}{\cos\theta}\kappa_\beta(t(s)),\quad \tau_\alpha(s)=-\sin\phi(s)\cos\phi(s)\,\kappa_\beta(t(s)).
\end{equation}
\end{prop}

On the other hand, from (\ref{curvaturas cilindro}) we get
\begin{equation*}
\frac{\tau_\alpha}{\kappa_\alpha}(s)=-\cos\theta\tan\phi(s),
\end{equation*}
and since $(\tan\phi)'(s)=\tan\theta\, \kappa_\beta(t(s))$, we have that only in circular cylinders there exist normal helices that are also rectifying curves, see \cite{Chen03}.

\subsection{An example: normal helices in circular cylinders}

Let $C_{\beta,V}$ be a cylinder over a circle $\beta$ of radius one. Then from Proposition \ref{hn-cil} we get
\begin{align*}
\phi(s)&=\arctan(\tan(\theta)s),\\
t(s)&=\cot(\theta)\sinh^{-1}(\tan(\theta)s)+t_0,\\
z(s)&=\cot(\theta)\sqrt{1+\tan^2(\theta)s^2}+z_0,
\end{align*}
where $\theta,t_0,z_0$ are constants. Hence, a family of normal helices in the cylinder is given by
\begin{align*}
\alpha(s)&=\Big(\cos(\cot(\theta)\sinh^{-1}(\tan(\theta)s)+t_0),\\
&\kern2em\sin(\cot(\theta)\sinh^{-1}(\tan(\theta)s)+t_0),\\
&\kern2em\cot(\theta)\sqrt{1+\tan^2(\theta)s^2}+z_0\Big).
\end{align*}
Moreover, from (\ref{curvaturas cilindro}) we obtain that the curvature and torsion of $\alpha$ are given by
\begin{equation}\label{curvcpnccilindro}
\kappa_\alpha(s)=\frac{\cos\theta}{\cos^2\theta+\sin^2(\theta)\,s^2},\quad  \tau_\alpha(s)=\frac{-\sin\theta\cos(\theta)\,s}{\cos^2\theta+\sin^2(\theta)\,s^2}.
\end{equation}
Note that these curves verify that $\tau_\al/\kappa_\al(s)=-\sin(\theta)s$, so they are rectifying curves.

We can reparametrize the curves $\alpha$ to obtain a simpler expression. Indeed, let us consider the change of parameter $\tan(\theta)s=\sinh(\tan(\theta)t)$, then
\begin{equation*}
\alpha(t)=\big(\cos(t+t_0),\sin(t+t_0),\cot(\theta)\cosh(\tan(\theta)t)+z_0\big),
\end{equation*}
and the curvature and torsion can be computed as follows
\begin{equation}\label{curvcpnccilindrot}
\kappa_\alpha(t)=\frac{\sec\theta}{\cosh^2(\tan(\theta)t)},\quad \tau_\alpha(t)=\frac{-\sinh(\tan(\theta)t)}{\cosh^2(\tan(\theta)t)}.
\end{equation}
A picture of a normal helix is shown in figure \ref{fig.ex}.

\begin{figure}
\begin{center}
\includegraphics[height=140mm,angle=-90]{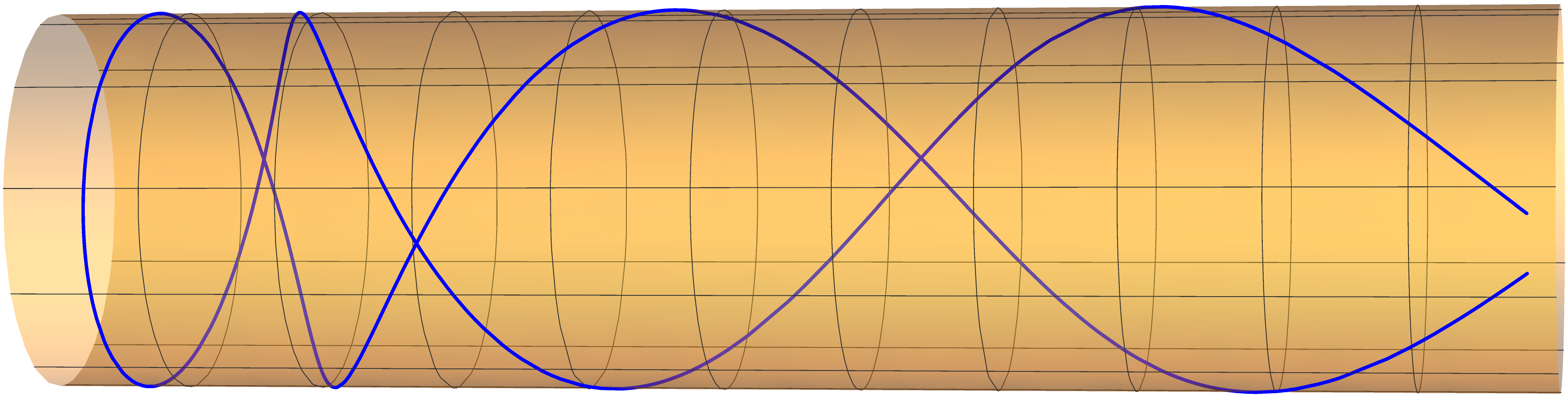}
\end{center}
\caption{\label{fig.ex}A normal helix with $\theta=\pi/36$ in a circular cylinder}
\end{figure}

\section{Osculating helices}

\subsection{Natural equation of osculating helices}
\label{s:oh}

Let $\alpha$ be a nonplanar osculating helix with axis $V$. Then there is an osculating vector field $W=\cos\theta\,T_\alpha+\sin\theta\,N_\alpha$, for a nonzero constant angle $\theta\in(-\pi/2,\pi/2)$, which is orthogonal to $V$. Hence we can write
\begin{equation}\label{VV}
V=\mu(-\sin\theta\;T_\alpha+\cos\theta\;N_\alpha)+\lambda\,B_\alpha,
\end{equation}
for certain differentiable functions $\lambda$ and $\mu$. By derivating here we obtain the following equations
\begin{align}
-\sin\theta\;\mu'-\mu\cos\theta\;\kappa_\alpha &=0, \label{eqV'11}\\
\cos\theta\;\mu'-\lambda\,\tau_\alpha-\mu\sin\theta\;\kappa_\alpha &=0, \label{eqV'22}\\
\lambda'+\mu\cos\theta\;\tau_\alpha &=0.\label{eqV'33}
\end{align}
From equation (\ref{eqV'11}) we get
\begin{equation}\label{mumu}
\mu=e^{-\cot\theta\,\int\kappa_\alpha},
\end{equation}
that jointly with (\ref{eqV'22}) yields
\begin{equation}\label{c}
\lambda=\frac{-1}{\sin\theta\rho}\, e^{-\cot\theta\,\int\kappa_\alpha},
\end{equation}
where $\rho$ is the Lancret curvature. On the other hand, by putting equations (\ref{mumu}) and (\ref{c}) in (\ref{eqV'33}), we get
\[
\frac{1}{\sin\theta}\left(\frac{\rho'}{\rho^2}+\cot\theta\,\kappa_\alpha\,\frac{\rho}{\rho^2}\right)= -\cos\theta\;\tau_\alpha,
\]
and then
\[
-\rho'=\sin\theta\cos\theta\,\tau_\alpha\rho^2+\cot\theta\,\kappa_\alpha\,\rho,
\]
which can be rewritten as
\begin{equation}\label{ecnaturalosculadora}
  \frac{\kappa_\alpha^2}{\kappa_\alpha^2+\sin^2\theta\,\tau^2_\alpha} \left(\frac{\tau_\alpha}{\kappa_\alpha}\right)'=-\cot\theta\,\tau_\alpha.
\end{equation}
Note that this equation is equivalent to
\begin{equation}\label{ecnaturalosculadora2}
\frac{\tau_\alpha}{\kappa_\alpha^2+\sin^2\theta\,\tau^2_\alpha} \left(\frac{\kappa_\alpha}{\tau_\alpha}\right)'=\cot\theta.
\end{equation}
Now, we will see that this equation characterizes the osculating helices. Indeed,
let $\alpha$ be a curve in $\R{3}$ satisfying (\ref{ecnaturalosculadora2}) for a nonzero constant $\theta\in(-\pi/2,\pi/2)$. Let us define a vector field $V$ as in (\ref{VV}), where $\lambda$ and $\mu$ are given by (\ref{c}) and (\ref{mumu}), respectively. Then equations (\ref{eqV'11}) and (\ref{eqV'22}) are satisfied. Finally, it is straightforward to see that equation (\ref{ecnaturalosculadora2}) leads to (\ref{eqV'33}). Then there exists a constant $\theta$ such that $\cos\theta\,T_\alpha+\sin\theta\,N_\alpha$ is orthogonal to the constant direction $V$, that is, $\alpha$ is an osculating helix.

We have proved the following characterization of the osculating helices (we also include the cases $\theta=\pm\pi/2$).

\begin{theo}
Let $\alpha$ be a nonplanar arclength parametrized curve in $\R3$, with curvature $\kappa_\alpha>0$ and torsion $\tau_\alpha$. Then $\alpha$ is an osculating helix (with $W$ orthogonal to $V$) if and only if the following equation holds
\begin{equation*}
 \frac{\tau_\alpha}{\kappa_\alpha^2+\sin^2\theta\,\tau^2_\alpha} \left(\frac{\kappa_\alpha}{\tau_\alpha}\right)'=\cot\theta,
\end{equation*}
for a nonzero constant angle $\theta\in[-\pi/2,\pi/2]$.
\end{theo}

\subsection{Normal helices and osculating helices}

Let $\alpha(s)$ be a normal helix with Frenet apparatous $\{\kappa_\alpha,\tau_\alpha; T_\alpha,N_\alpha,B_\alpha\}$. Then the curve
\def\balpha{{\overline\alpha}}
\begin{equation}\label{hosculadora}
\balpha(s)=\int_{s_0}^s B_\alpha(t)\,dt
\end{equation}
is an arclength parametrized curve, and without loss of generality its Frenet frame is given by
\[
T_\balpha=B_\alpha,\quad
N_\balpha=-N_\alpha,\quad
B_\balpha=T_\alpha.
\]
Hence their curvature and torsion are given by $\kappa_\balpha=\tau_\alpha$ and $\tau_\balpha=\kappa_\alpha$. It is straightforward to see that $\balpha$ satisfies (\ref{ecnaturalosculadora2}), and so it is an osculating helix.

On the other hand, and following a similar reasoning, it can be proved that if $\balpha$ is an osculating helix then the curve
\begin{equation}\label{hnormal}
\alpha(s)=\int_{s_0}^s B_\balpha(t)\,dt
\end{equation}
is a normal helix. Therefore, and in a certain sense, normal helices and osculating helices can be considered as duals of each other.

\section{Rectifying helices}
\label{s:rh}

Let $\alpha$ be a nonplanar curve in $\R3$ with Frenet apparatous $\{\kappa_\alpha,\,\tau_\alpha;T_\alpha,N_\alpha,B_\alpha\}$, and assume that $\alpha$ is a rectifying helix with axis $V$, $V$ being a constant vector. Let us suppose there is a constant angle $\theta$ such that the rectifying vector field $W=\cos\theta\, B_\alpha+\sin\theta\,T_\alpha$ is orthogonal to $V$. Hence we can write
\begin{equation*}
V=\lambda\,N_\alpha+\mu(\sin\theta\,B_\alpha-\cos\theta\,T_\alpha),
\end{equation*}
for certain differentiable functions $\lambda$ and $\mu$. By taking derivative there we get
\begin{align*}
-\lambda\,\kappa_\alpha-\mu'\,\cos\theta &=0,\\
\lambda'-\mu\,\cos\theta\,\kappa_\alpha-\mu\sin\theta\,\tau_\alpha &=0,\\
\lambda\,\tau_\alpha+\mu'\,\sin\theta&=0.
\end{align*}
From these equations we easily deduce that $\tau_\alpha/\kappa_\alpha$ is constant, and so $\alpha$ is a cylindrical helix. Since every cylindrical helix is also a rectifying helix, then we have proved the following result.
\begin{theo}
Let $\alpha$ be an arclength parametrized curve in $\R3$ with curvature $\kappa_\alpha>0$. Then $\alpha$ is a rectifying helix (with $W$ orthogonal to $V$)  if and only if it is a cylindrical helix.
\end{theo}

\section{The general case}
\label{s:gsh}

Let us assume, in this section, that the unit \svf{} vector field $W$ along $\alpha$ is given by $W=aT_\alpha+bN_\al+cB_\al$, where $a,b,c$ are nonzero constants with $a^2+b^2+c^2=1$ (the other cases have already been analyzed in the preceding sections). Since we are assuming that $\<W,V\>=0$, then there exists two differentiable functions $\lambda,\mu$ such that
\begin{equation}\label{Vgen}
V=\lambda(-cT_\al+aB_\al)+\mu(-cN_\al+bB_\al),
\end{equation}
and by derivating here we obtain the following equations:
\begin{align}
0&= \lambda'-\mu\kappa_\al,\label{4.1}\\
0&= c\mu'+\lambda(c\kappa_\al+a\tau_\al)+b\mu\tau_\al,\label{4.2}\\
0&= a\lambda'+b\mu'-c\mu\tau_\al.\label{4.3}
\end{align}
From (\ref{4.1}) and (\ref{4.3}) we get
\begin{equation}\label{4.4}
\mu'=\frac1b(c\tau_\al-a\kappa_\al)\mu,
\end{equation}
that jointly with (\ref{4.2}) leads to
\[
c\mu(c\tau_\al-a\kappa_\al)+\lambda b(c\kappa_\al+a\tau_\al)+b^2\mu\tau_\al=0.
\]
Then
\begin{equation}\label{4.45}
\lambda=g\mu,\quad\text{ with } g=\frac{ac\kappa_\al-(b^2+c^2)\tau_\al}{b(c\kappa_\al+a\tau_\al)}.
\end{equation}
Now, by using (\ref{4.1}) and (\ref{4.4}) we obtain
\begin{equation}\label{4.5}
b\kappa_\al=bg'+g(c\tau_\al-a\kappa_\al).
\end{equation}
Since $g'$ is given by
\[
g'=\frac{-c\kappa_\al^2}{b(c\kappa_\al+a\tau_\al)^2}\left(\frac{\tau_\al}{\kappa_\al}\right)',
\]
a straightforward computation from (\ref{4.5}) yields
\begin{equation}\label{4.7}
\frac{\kappa_\al^2} {c\kappa_\al\big((1-c^2)\kappa_\al^2+(1-3a^2)\tau_\al^2\big)+ a\tau_\al\big((1-a^2)\tau_\al^2+(1-3c^2)\kappa_\al^2\big)}\left(\frac{\tau_\al}{\kappa_\al}\right)'=-\frac1b.
\end{equation}
Conversely, let $\alpha$ be a curve satisfying (\ref{4.7}) for certain nonzero constants $a,b,c$. Let $V$ be the nonzero vector field given in (\ref{Vgen}), where $\mu$ and $\lambda$ are given by (\ref{4.4}) and (\ref{4.45}), respectively. As in the preceding sections, it is a straightforward (if somewhat laborious) calculation to check that equations (\ref{4.1})--(\ref{4.3}) are satisfied. Therefore, we have found a \svf{} vector field $W=aT_\alpha+bN_\al+cB_\al$ which is orthogonal to the fixed direction $V$.

To finish this section, and in order to consider also the cases in which any of the constants $a$, $b$ or $c$ could be zero, let us note that  equation (\ref{4.7}) can be rewritten as
\begin{equation}\label{4.7'}
-b\kappa_\al^2\left(\frac{\tau_\al}{\kappa_\al}\right)'=
c\kappa_\al\big((1-c^2)\kappa_\al^2+(1-3a^2)\tau_\al^2\big)+ a\tau_\al\big((1-a^2)\tau_\al^2+(1-3c^2)\kappa_\al^2\big),
\end{equation}
and then we have all the cases previously analyzed:
\begin{center}\def\arraystretch{1.2}
\begin{tabular}{llc}\hline
$W$ & $\alpha$ is a & Eq. (\ref{4.7'}) reduces to\\\hline
\rule{0pt}{15pt}$T_\al$ & plane curve & $\rho=0$\\
$N_\al$ & cylindrical helix & $\rho'=0$\\
$B_\al$ & plane curve & $\rho=0$\\
$bN_\al+cB_\al$ & normal helix & Eq. (\ref{ecnaturalnormal})\\
$aT_\al+bN_\al$ & osculating helix & Eq. (\ref{ecnaturalosculadora2})\\
$aT_\al+cB_\al$ & rectifying helix & $\rho'=0$\\
$aT_\alpha+bN_\al+cB_\al$ & helix & Eq. (\ref{4.7})\\\hline
\end{tabular}
\end{center}
In conclusion, we have shown the following result.
\begin{theo}
Let $\alpha$ be a nonplanar arclength parametrized curve in $\R3$ with curvature $\kappa_\alpha>0$ and torsion $\tau_\alpha$. Then $\alpha$ is a helix (associated to the unit \svf{} vector field $W=aT_\alpha+bN_\al+cB_\al$ orthogonal to the axis) if and only if the following equation holds
\[
-b\kappa_\al^2\left(\frac{\tau_\al}{\kappa_\al}\right)'=
c\kappa_\al\big((1-c^2)\kappa_\al^2+(1-3a^2)\tau_\al^2\big)+ a\tau_\al\big((1-a^2)\tau_\al^2+(1-3c^2)\kappa_\al^2\big),
\]
with $a^2+b^2+c^2=1$.
\end{theo}

We finish this section with the geometric interpretation of helices in the general case.
\begin{theo}\label{th4}
Let $\alpha$ be a nonplanar arclength parametrized curve in $\R3$ with curvature $\kappa_\alpha>0$ and torsion $\tau_\alpha$. Then $\alpha$ is a helix (associated to the unit \svf{} vector field $W=aT_\alpha+bN_\al+cB_\al$ orthogonal to the axis $V$) if and only if $\al$ is contained in a cylinder $C_{\beta,V}$ and satisfies the following equation
\begin{equation}\label{th4.eq}
a\tan\phi+b\sin\theta-c\cos\theta=0,
\end{equation}
where $a,b,c$ are real constants and the angles $\phi$ and $\theta$ are given by
\[
\sin\phi=\<T_\al,V\>,\quad\sin\theta\<B_\al,V\>+\cos\theta\<N_\al,V\>=0.
\]
\end{theo}
\begin{proof}
We will follow the same reasoning as for proving Theorem \ref{th3}. Let us first prove the reciprocal part. Let $X(t,z)=\beta(t)+zV$ be the canonical parametrization of $C_{\beta,V}$ and suppose $\al(s)=X(t(s),z(s))$. Although equations (\ref{TNB-cil}) and (\ref{TNB-cil2}) have been obtained for the case where $\theta$ is constant, they also remain valid in the non-constant case, and so the $F$-constant vector field $W=aT_\al+bN_\al+cB_\al$ satisfies $\<V,W\>=0$. Hence $\al$ is a helix with axis $V$ and $F$-constant vector field $W$ orthogonal to it.

Now, let us consider $\al$ a helix with axis $V$, and suppose there is a $F$-constant vector field $W=aT_\al+bN_\al+cB_\al$ satisfying $\<V,W\>=0$. Then $\al$ is contained in the cylinder $C_{\al,V}$, which can be locally parametrized by $X(t,z)=\beta(t)+zV$, where $\beta(t)$ is a plane curve in $C_{\al,V}$ orthogonal to $V$. Then equations (\ref{TNB-cil}) and (\ref{TNB-cil2}) are satisfied, and the condition $\<V,W\>=0$ implies equation (\ref{th4.eq}), since $\cos\phi\neq0$ (otherwise, $\al$ would be a plane curve).
\end{proof}

Following a reasoning similar to the one used to prove Proposition \ref{hn-cil}, we can obtain from Theorem \ref{th4} the following result (which generalizes Proposition \ref{hn-cil} to the case where $\theta$ is not constant). We leave the proof to the reader.

\begin{prop}\label{cg-cil}
Let $\alpha(s)=X(t(s),z(s))$ be an arclength parametrized curve in a cylinder $C_{\beta,V}$. Then $\al$ satisfies equation (\ref{th4.eq}), for a non constant function $\theta\equiv\theta(s)$, if and only if
\begin{align*}
  t'(s)&=\frac{a}{\sqrt{a^2+(-b\sin\theta+c\cos\theta)^2}},\\
  z'(s)&=\frac{-b\sin\theta+c\cos\theta}{\sqrt{a^2+(-b\sin\theta+c\cos\theta)^2}},\\
  \kappa_\beta(t(s))&=-\frac{(b\cos\theta+c\sin\theta)\theta'}{a\tan\theta}.
\end{align*}
\end{prop}

\section*{Acknowledgement}

This research is part of the grant PID2021-124157NB-I00, funded by MCIN/ AEI/ 10.13039/ 501100011033/ ``ERDF A way of making Europe''. Also supported by ``Ayudas a proyectos para el desarrollo de investigación científica y técnica por grupos competitivos'', included in the ``Programa Regional de Fomento de la Investigación Científica y Técnica (Plan de Actuación 2022)'' of the Fundación Séneca-Agencia de Ciencia y Tecnología de la Región de Murcia, Ref. 21899/PI/22.

\end{document}

%% file: math.tex
\newcommand{\bc}{\begin{center}}
\newcommand{\ec}{\end{center}}

\numberwithin{equation}{section}

\newcommand{\al}{\alpha}

\renewcommand{\phi}{\varphi}